\documentclass{amsart}

\usepackage[colorlinks=true, urlcolor=NavyBlue, linkcolor=NavyBlue, citecolor=NavyBlue]{hyperref}

\usepackage{graphicx}
\usepackage[percent]{overpic}

\usepackage[dvipsnames,usenames]{color}
\usepackage{tikz}

\usepackage{subfigure}

\usepackage{amsmath,amsthm,amscd,amssymb}
\newtheorem{thm}{Theorem}[section]
\newtheorem{lem}[thm]{Lemma}
\newtheorem{prop}[thm]{Proposition}

\theoremstyle{definition}
\newtheorem{exam}[thm]{Example}
\newtheorem{defn}[thm]{Definition}

\numberwithin{equation}{section}
\numberwithin{figure}{section}

\newcommand{\inv}{^{-1}}

\newcommand{\del}{\partial}

\newcommand{\bbr}{{\mathbb{R}}}
\newcommand{\bbz}{{\mathbb{Z}}}
\newcommand{\bbq}{{\mathbb{Q}}}

\newcommand{\rank}{{\mathrm{rank}}}

\title{On the intersection ring of graph manifolds}

\author{Margaret I. Doig}
\address{Department of Mathematics\\Syracuse University\\215 Carnegie Building\\Syracuse, NY 13244-1150}
\email{midoig@syr.edu}
\urladdr{http://midoig.mysite.syr.edu/}

\author{Peter D. Horn$\dagger$}
\thanks{$\dagger$ Partially supported by National Science Foundation DMS-1258630}
\address{Department of Mathematics\\Syracuse University\\215 Carnegie Building\\Syracuse, NY 13244-1150}
\email{pdhorn@syr.edu}
\urladdr{http://pdhorn.expressions.syr.edu/}

\begin{document}

\maketitle

\begin{abstract}
	We calculate the intersection ring of three-dimensional graph manifolds with rational coefficients and give an algebraic characterization of these rings when the manifold's underlying graph is a tree. We are able to use this characterization to show that the intersection ring obstructs arbitrary three-manifolds from being homology cobordant to certain graph manifolds.
\end{abstract}

\section{Introduction}

The motivating problem for this paper is the homology cobordism classification of $3$-manifolds. Two closed, oriented $3$-dimensional manifolds $M$ and $N$ are \emph{homology cobordant} if there exists a $4$-dimensional manifold $W$ with $\del W = M \sqcup -N$ and the inclusion maps from $M$ and from $N$ into $W$ inducing isomorphisms on homology.  Homology cobordism is an equivalence relation, and $\Theta^3_\bbz$ is the group of homology cobordism classes of $3$-manifolds under the connected sum operation.  Livingston proved that every class $[M]\in\Theta^3_\bbz$ is represented by an irreducible $3$-manifold~\cite[Theorem 3.2]{Liv:HomCob}. Myers refined Livingston's result and proved that every class $[M]\in\Theta^3_\bbz$ is represented by a hyperbolic $3$-manifold~\cite{Myers}.

Since the $3$-manifolds with hyperbolic geometries form a rich class of examples, one might look for the simplest class of $3$-manifolds representing all homology cobordisms classes. Cochran and Tanner took up this question and proved, using the cohomology ring and Massey products -- two invariants of homology cobordism -- that there exist infinitely many distinct $3$-manifolds, none of which is homology cobordant to any Seifert fibered $3$-manifold \cite{CocTan}.  The class of \emph{graph manifolds} includes all Seifert fibered manifolds, and, generally, a graph manifold's cohomology ring and Massey products are less restrictive than those of a Seifert fibered space.

This paper arose out of our attempt to determine whether every $3$-manifold is homology cobordant to a graph manifold; this matter remains open.

Let $M$ be a graph manifold. In this paper, we describe an explicit set of curves in $M$ that generate $H_1(M;\bbq)$; we construct explicit surfaces in $M$ that generate $H_2(M;\bbq)$ and that are dual, in the sense of intersection, to the curves generating $H_1(M;\bbq)$; and we describe in generality the intersection product among these surfaces (cf. Theorem~\ref{thm:IntersectionRing}). This is equivalent to describing the cohomology ring with rational coefficients of an arbitrary graph manifold. Perhaps it is surprising that the cohomology rings of graph manifolds have not been classified, though Aaslepp, Drawe, Hayat-Legrand, Sczesny and Zieschang have computed the cochain complex and most of the cohomology ring for most graph manifolds, with coefficients modulo $2$ \cite{ADHLSZ}.

For those graph manifolds whose underlying graph is a tree (called tree graph manifolds), the cohomology ring is slightly simpler than in the general case.  Not only do we describe these rings in terms of curves and surfaces and their intersections, but we also characterize these rings using an algebraic construction called the connected sum of rings.  While even these rings are difficult to recognize, we use the cohomology ring invariant to obstruct homology cobordism to tree graph manifolds.

\newtheorem*{thm:result} {Theorem~\ref{thm:TreeGraphManifoldsNotGeneric}}
\begin{thm:result}
	Not every closed, compact $3$-manifold is homology cobordant to a tree graph manifold.
\end{thm:result}

There are several lines of inquiry to pursue in the future, including the extension of Theorem~\ref{thm:TreeGraphManifoldsNotGeneric} to the class of (general) graph manifolds.  The higher Massey products of graph manifolds might be a fruitful avenue to pursue; we considered the triple Massey product and were unable to conclude anything the cohomology ring would not already establish.

\subsection*{Organization} We review Seifert fibered spaces and graph manifolds, as well as set conventions in Section~\ref{section:definitions}.  In Sections~\ref{section:homology} and~\ref{section:cohomology}, we describe the homology groups and their ring structure, given by the intersection product.  We give an algebraic characterization of the intersection ring of tree graph manifolds in Section~\ref{section:TreeManifolds} and exhibit a $3$-manifold in Section~\ref{section:ExampleAwesome} that is not homology cobordant to a tree graph manifold.

\section{Definitions}\label{section:definitions}

We review the definition of a \emph{Seifert fibered space}, a space fibered by circles in an appropriate fashion, and then a \emph{graph manifold}, a manifold whose JSJ decomposition includes only Seifert fibered pieces. We will later define a special type of graph manifold, a \emph{tree graph manifold}.

\begin{defn} A \emph{Seifert fibered space} is a manifold fibered by $S^1$ where every fiber has a neighborhood which is covered by the solid torus $D^2\times S^1$ with fibration $*\times S^1$ so that the covering map respects the fibration.
\end{defn}

To construct a Seifert fibered space, consider $S$ be an orientable or nonorientable surface with boundary components $c_1,\ldots,c_q$ and $d_1,\ldots, d_p$.  Assume there is at least one boundary component for simplicity (we do want $S^1 \times S^2$ to be a Seifert fibered space, but we can arrange this in the next step by doing an appropriate Dehn filling).  Let $N$ be the unique circle bundle over $S$ with orientable total space, corresponding to the trivial homomorphism $\pi_1(S)\to \bbz_2$.  Let $t$ denote a fiber in $N$, and call $t$ a \emph{regular fiber}.  For each torus boundary component $t \times c_i$, choose a pair of relatively prime integers $a_i$ and $b_i$ so that $0 < a_i < |b_i|$ and  glue in a solid torus so that the meridional curve is identified with $b_i \, c_i + a_i\, t$.  We call this a \emph{$b_i/a_i$-framed Dehn filling}, and we call the core of the surgery solid torus a \emph{critical fiber}.  There are $p$ remaining boundary tori $d_i \times t$.  After these steps one has a Seifert fibered 3-manifold, $Y$.

As a matter of convention, we denote the genus of the nonorientable surface ${\mathop{\Large {\#}}}_{i=1}^g \bbr P^2$ by $-g$.

\begin{defn}
	An orientable, closed $3$-manifold $M$ is a \emph{graph manifold} if there is a collection $T$ of tori with product neighborhoods so that $M - N(T)$ is a collection of disjoint Seifert fibered manifolds.
\end{defn}

A graph manifold can be obtained by gluing oriented Seifert fibered spaces along their boundary tori by homeomorphisms.  We orient the Seifert fibered pieces and glue them by orientation-reversing homeomorphisms of their boundary tori, which orients the graph manifold. These homeomorphisms of the torus can be identified with the component of $GL_2(\bbz)$ of negative determinant.  Let $\begin{pmatrix}
	a&b\\c&d
\end{pmatrix}$ correspond to the homeomorphism of the torus taking an identified meridian $\mu$ and longitude $\lambda$ to $a\mu + c\lambda$ and $b\mu + d\lambda$, respectively.  Therefore any graph manifold can be represented by a decorated graph whose vertices are Seifert fibered spaces and whose edges show which Seifert fibered spaces are glued to which. Decorations must accompany the nodes (Seifert data) and edges (gluing homeomorphisms). To simplify calculations later, we make three assumptions:

	\begin{enumerate}
		\item Each gluing homeomorphism is either $J = \begin{pmatrix} 0& 1\\ 1&0 \end{pmatrix}$ or $-J$, and the edges are adorned with $\pm 1$ accordingly.
		\item No vertex has a self-loop (resolve a self loop by replacing the edge with an edge-vertex-edge-vertex-edge sequence, each vertex with trivial Seifert fibered space $T^2 \times I$).
		\item All graphs are connected (equivalently, all manifolds are prime).
	\end{enumerate}

To show the first assumption is not an assumption at all, we briefly describe how to reduce an arbitrary gluing homeomorphism to a plumbing homeomorphism. (This is mathematical folklore, but the authors are unable to find a reference in the literature.)  This essentially is the proof that graph manifolds are equivalent to plumbed manifolds, each of which were classified by Waldhausen~\cite{Waldhausen:GraphI},\cite{Waldhausen:GraphII} and Neumann~\cite{Neumann}, respectively.  We will calculate these manifolds' cohomology rings using their plumbing structures.

\begin{lem}\label{lem:torushomeo}
	Let $X = T^2 \times I$ and $U = 1 \times S^1 \times 1/2 \subset X$.  Then $X \cong X_{1/n}(U)$, where the latter is the result of $1/n$-framed Dehn surgery on $U$ as defined above.
\end{lem}
\begin{proof}
	The standard proof that $1/n$-framed Dehn surgery on the unknot in $S^3$ is homeomorphic to $S^3$ relies on doing a Dehn twist repeatedly on the complement of the unknot to reduce the surgery coefficient to $1/0$.  Embed $X$ into $S^3$ taking $U$ to the unknot, and $X$ is invariant under this Dehn twist.  Thus, $X_{1/n}(U) \cong X_{1/n-1}(U)\cong \cdots\cong X_{1/0}(U) \cong X$.
\end{proof}

\begin{lem}\label{lem:ColOp}
	Let $Y$ denote the gluing of Seifert fibered spaces $M$ and $N$ by a homeomorphism $\begin{pmatrix}
		a&b\\c&d
	\end{pmatrix}$.  Then $Y$ is homeomorphic to the gluing of Seifert fibered spaces $M$ and $N'$ by a homeomorphism 	$\begin{pmatrix}
			a-nb & b \\ c-nd & d
		\end{pmatrix}$, where $N'$ is obtained from $N$ by adding one critical fiber of type $1/n$.
\end{lem}
\begin{proof}
	By Lemma~\ref{lem:torushomeo}, the described $N'$ is homeomorphic to $N$, and it is well-known that the composition of Dehn twists acts on the gluing homeomorphism as claimed.
\end{proof}

Lemma~\ref{lem:ColOp} describes how to reduce a gluing homeomorphism by a column operation on a particular matrix. The column operation from left to right may be achieved by altering $M$ instead of $N$. If a nontrivial gluing map occurs in a graph manifold, then applying this reduction repeatedly alters the matrix to
\[ \begin{pmatrix}
	0 & \pm\mathrm{gcd}(a,b)\\
	\pm \mathrm{gcd(c,d)} & 0
\end{pmatrix} \] which is  $\pm J$ since $\det \begin{pmatrix}
		a&b\\ c&d
	\end{pmatrix} = -1$.

\section{Homology of graph manifolds}\label{section:homology}

We schematically represent a graph manifold $M$ by a graph as in Figure~\ref{fig:graphmanifold}.

\begin{figure}[!ht]
	\begin{center}
		\begin{overpic}{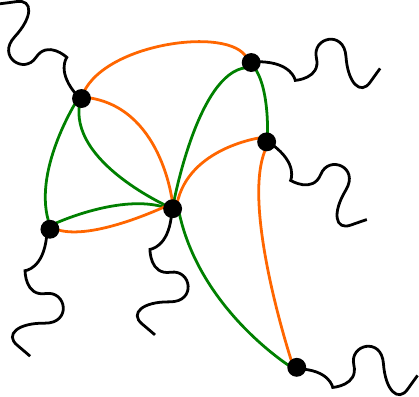}
			\put(100,50){\textcolor{BurntOrange}{$T$: maximal tree}}
			\put(100,40){\textcolor{ForestGreen}{$E$: edges not in maximal tree}}
		\end{overpic}
		\caption{Our schematic graph $G$ of the graph manifold $M$}
		\label{fig:graphmanifold}
	\end{center}
\end{figure}

Each node with its attached squiggle in Figure~\ref{fig:graphmanifold} represents a Seifert fibered piece of $M$.  We choose to think of the squiggle as the essential portion of the Seifert fibered piece, which contains all the genus of the base and all of the critical fibers, and we think of each node in Figure~\ref{fig:graphmanifold} as a punctured $S^2$ cross $S^1$.  With this scheme, all of the plumbings between the Seifert fibered components of $M$ happen between $P\times S^1$'s, where $P$ is a punctured $S^2$.  The squiggles are attached to their nodes by identifying longitude to longitude and meridian to meridian.

There is some back-and-forth between graphs and manifolds.  By \emph{node}, we mean a vertex in the graph $G$. By \emph{end}, we mean a Seifert fibered submanifold (depicted by a squiggle in Figure~\ref{fig:graphmanifold}) of the manifold $M$.  To clarify things, when necessary, for a node (or subgraph) $N$, we will write $gN$ to denote the associated vertex (or subgraph) of $G$, and we will write $mN$ to denote the associated submanifold of $M$.  Under this convention, $mG = M$.

Unless stated, all homology and cohomology in this section is taken with rational coefficients.

\subsection{The plumbing graph} 

Let $N$ denote a node of $G$, so $mN \cong P \times S^1$ where $P$ is a $p$-times-punctured $S^2$.  Then $H_1(mN)$ is free of rank $p$ with generators $t, d_1,\ldots, d_p$ and relation $\sum d_i = 0$.  Here $t$ is carried by the fiber $S^1$, and the $d_i$ are carried by the boundary components of $P$.  We enumerate the $d_i$ so that the last, $d_p$, is the boundary component which is glued by the identity to the Seifert fibered piece (`squiggle') attached to the node $N$.

Ignore the squiggle ends in $G$. Let $T$ be a maximal tree for the remaining graph and $E$ be the remaining edges not in $T$. We aim to calculate $H_1(M)$ by
\begin{enumerate}
	\item calculating $H_1(mT)$,
	\item calculating $H_1(mT\cup mE)$, and
	\item calculating $H_1(mG) = H_1(M)$
\end{enumerate}
by a sequence of Mayer-Vietoris calculations.

\begin{lem}
	Let $N_1$ and $N_2$ be two nodes joined in $G$ by an edge in $T$ by $\pm J$.  Then $H_1(mN_1\cup mN_2)$ is free of rank $\beta_1(mN_1)+\beta_1(mN_2)-2$.  It is generated by $t^1, d^1_k$ and $t^2, d^2_k$ (the regular fiber and \emph{all} boundary components of $N_1$ and $N_2$) with relations:
	\begin{enumerate}
		\item $d^i_{p_i} = - \sum_{j=1}^{p_i-1} d^i_j$, for $i=1,2$
		\item $d^1 = \pm t^2$
		\item $d^2 = \pm t^1$
	\end{enumerate}
\end{lem}

The proof is immediate, and it generalizes to

\begin{lem}\label{lem:homologytree}
	Let $T$ denote a maximal tree in $G$ with nodes $N_1, N_2, \cdots, N_n$.  Then $H_1(mT)$ is free of rank $n+2\#\, E$ with generators $t^i, d^i_k$ (again, \emph{all} boundary components) and relations:
	\begin{enumerate}
		\item $d^i_{p_i} = -\sum_{j=1}^{p_i-1} d^i_j$ for each $i$
		\item $d^i = \pm t^j$ and $d^j = \pm t^i$ whenever nodes $i$ and $j$ are joined by an edge decorated by $\pm 1$ in $T$
	\end{enumerate}
	The $t^i$ and $d^i_k$, for those $d^i_k$ that are not involved in any plumbings in the tree $T$ nor in plumbing $T$ to its Seifert fibered pieces, are free generators of $H_1(mT)$.
\end{lem}

Note that we have chosen orientations on the regular fibers of each node $mN_i$, so we must take care in the plumbings associated to the edges.  Each plumbing has an associated sign $\pm$ according to whether the plumbing identifies a positively oriented regular fiber with a $\pm$-oriented boundary component from the other node.

It is straightforward that each plumbing from an edge in $E$ reduces the rank of $H_1(mG)$ by $1$: the plumbing  equates two boundary component generators from Lemma~\ref{lem:homologytree} with $\pm$ regular fibers, and it creates a new generator corresponding to the loop in $gG$ arising from adding an edge to the maximal tree.

\begin{lem}\label{lem:homgraph}
	Let $b = \beta_1(G) = \#\, E$.  Then $H_1(mT\cup mE)$ is free of rank $b + \#\,\mathrm{nodes}$ with generators consisting of regular fibers of the nodes, boundary components in the nodes, and loops corresponding to the edges in $E$.  Each $d^i_{p_i} = -\sum_{j=1}^{p_i -1} \pm t^{k(j)}$, where $d^i_j$ is plumbed by $\pm J$ to node $k(j)$.
\end{lem}

\subsection{The squiggle ends, or Seifert fibered pieces}

To calculate $H_1(M)$, it remains to attach the Seifert fibered ends to $T\cup E$.  Note that each vertex in $T\cup E$ gives one boundary component $T^2$, and so each Seifert piece we attach has one boundary component.

Let $N$ denote a Seifert fibered space with base a surface $S$ of genus $g$ with $q$ exceptional fibers of framing $b_1/a_1,\ldots, b_q/a_q$, and one boundary component.  Our convention is that the meridinal disk of the $i$th exceptional torus has boundary $b_i[\mu_i] + a_i[\lambda_i] \in \pi_1(T^2_i)$, and $0 < a_j < |b_j|$.  Here each $\lambda_i$ is identified with a regular fiber.

If $S$ is orientable of genus $g$ with one boundary curve $d$, recall that
\[\pi_1(N) \cong \langle \alpha_i, \beta_i, c_j, d, t:\ [\alpha_i,t], [\beta_i,t], [c_j,t], [d_k,t], c_j^{b_j}t^{a_j}, [\alpha_1,\beta_1]\cdots[\alpha_g,\beta_g] c_1 \cdots c_q d \rangle.\]
If $S$ is non-orientable of genus $g$,
\[\pi_1(N)\cong \langle \delta_i, c_j, d, t:\ \delta_i t \delta_i^{-1} t, [c_j,t], [d,t], c_j^{b_j}t^{a_j}, \delta_1^2\cdots \delta_g^2 c_1\cdots c_q d \rangle.\] 
In either case, $t$ is carried by a regular fiber (see, e.g.,~\cite[p. 117]{Hempel}).

\begin{lem}\label{lem:n-o-torsion}
If $S$ is nonorientable, then $H_1(N;\bbz)$ contains an element of order $2$ carried by a regular fiber, and the boundary component $d$ of $S$ has infinite order. $H_1(N;\bbq) \cong \bbq^g$ with basis $\delta_i$ where $d = -2\sum \delta_i$.
\end{lem}
\begin{proof}
	Note that $H_1(N;\bbz)$ has the $\bbz$-module presentation with generators $\delta_i, c_j, d,$ and $t$ and relations
	\begin{eqnarray*}
		2\,t &=& 0, \\
		b_j\,c_j &=& -a_j\,t,\\
		2\sum \delta_i + \sum c_j + d &=& 0.		
	\end{eqnarray*}
There are a total of $1+q+g+1$ generators.  Tensoring with $\bbq$ shows that the rank of $H_1(N;\bbz)$ is $g$ and that $d$ has infinite order in $H_1(N;\bbz)$.
	
	Suppose for the sake of contradiction that $H_1(N;\bbz)$ contains no $2$-torsion.  Then $H_1(N;\bbz) \cong \bbz^{g} \oplus (\mathrm{odd\ torsion})$, so $H_1(N;\bbz)\otimes \bbz_2$ has the same rank as $H_1(N;\bbz)$.  However, reducing the relations modulo $2$ yields
	\begin{eqnarray*}
		b_j\, c_j &=& a_j\,t,\\
		\sum c_j + d &=& 0,
	\end{eqnarray*}
and so $H_1(N;\bbz)\otimes \bbz_2$ has a presentation with $1+q+g+1$ generators and at most $q+1$ relations.  Thus $\rank(H_1(N)\otimes \bbz_2) \geq (1+q+g+1) - (q+1) = g+1$, which is a contradiction.
	
	Suppose $t$ does not have order $2$.  Then $2\,t=0$ implies $t=0$, which eliminates a generator and a relation.  Then, by the last relation, some $c_j$ is redundant, so we may eliminate that generator, the last relation, and the corresponding relation $b_j\,c_j = 0$.  Then we have $1+g+q+1 - 2$ generators and $q+2 - 3$ relations, so the rank of $H_1(N;\bbz)$ is at least $(1+g+q+1-2) - (q + 2 - 3) = g+1$, which is a contradiction.
\end{proof}

\begin{lem}\label{lem:o-rational}
If $S$ is orientable, then the regular fiber $t$ and boundary component $d$ have infinite order in $H_1(N;\bbz)$. $H_1(N;\bbq)\cong \bbq^{2g+1}$ with basis $\alpha_i, \beta_i,$ and $t$ and where $d=-2\sum c_j$.
\end{lem}

\begin{proof}
Note that $H_1(M)$ has the $\bbz$-module presentation with generators $\alpha_i,\beta_i,c_j,d$, and $t$ and relations
	\begin{eqnarray*}
		b_j\,c_j &=& -a_j\,t,\\
		\sum c_j + d_k &=& 0.
	\end{eqnarray*}
	
	The equation $n\,t=0$, if true, must be a consequence of the above equations.  Since $d$ only appears in the last equation, that equation cannot be involved in proving $n\,t=0$.  Since a different $c_j$ appears in each of the other equations, we conclude $n=0$.  Thus $t$ has infinite order.
	
	With rational coefficients, the $c_j$ are redundant, as is $d$.  The $\alpha_i$, $\beta_i$, and $t$ are essential, and the conclusion follows.
\end{proof}

\subsection{Gluing the Seifert fibered squiggle ends to the graph}

To attach a Seifert fibered end to a node in the graph, we identify the regular fiber of the end with the regular fiber of the node and the boundary component of the end with the boundary component of the node. The generators for $H_1(M;\bbq)$ are:
\begin{itemize}
	\item $\delta_1, \ldots, \delta_g$ for each nonorientable base end,
	\item $\alpha_1,\beta_1, \ldots, \alpha_g, \beta_g$ for each orientable base end,
	\item $t$ for each node, except those nodes to which a nonorientable base end is attached, and
	\item $\gamma$ for each edge in $E$, i.e., $\beta_1(G)$.
\end{itemize}
By Lemmata~\ref{lem:homgraph} and~\ref{lem:n-o-torsion}, each end with nonorientable base produces a relation \(2\sum \delta_i = \sum \pm t^k\) over the $t^k$ for all adjacent nodes. By Lemmata~\ref{lem:homgraph} and~\ref{lem:o-rational}, each end with orientable base produces a relation \( \sum a_j/b_j\, t = \sum \pm t^k \) again over $t^k$ for all adjacent nodes. These are the only relations.

For each nonorientable base end, we choose to eliminate the generator $\delta_g$ and the relation \(2\sum \delta_i = \sum \pm t^k\). By some linear combination of the relations, we may be able to eliminate some of the regular fiber generators. To determine which regular fibers are essential and which are not, consider the \emph{orientable subgraph} $H$, the subgraph of $G$ obtained by removing all nonorientable base ends and their nodes.  Let $A$ denote the \emph{connectivity matrix} of $H$:
\begin{itemize}
	\item order the vertices of $H$,
	\item in the diagonal entry $(i,i)$, put the quantity $-\sum a_j/b_j$ coming from the exceptional fibers of the $i$th end,
	\item in the off-diagonal entry $(i,j)$, put the signed number of edges in $H$ connecting node $i$ to node $j$.
\end{itemize}

These are the only relations that can kill a regular fiber (or linear combination thereof) in $H_1(M)$.  $A$ is a symmetric matrix and may be diagonalized to $\mathrm{diag}{(A)}$.  The number of diagonal zeros in $\mathrm{diag}(A)$ is the number of free generators of $H_1(M)$ coming from regular fibers, which we call the \emph{number of regular fibers that survive in $H_1(M)$}.  If we record how we diagonalize $A$, we can specify which regular fibers are left as free generators and which are combinations of these generators (cf. Example~\ref{exam:SurvivingFiberFromTwoNodes}).

\begin{thm}\label{thm:homologyofgraphmanifold}
	Let $M$ be a graph manifold with graph $G$.  Let $r$ denote the number of regular fibers that survive.  Let $b = \beta_1(G)$.  Let $g_- = \sum_B g(B)-1$ where $B$ ranges over the bases of the nonorientable base ends of $M$ and where $g(B)$ is the number of summands in $B = P^2\#\cdots\# P^2$.  Let $g_+ = \sum_S g(S)$, where $S$ ranges over the bases in the orientable base ends of $M$.  Then 
\[ H_1(M;\bbq) \cong \bbq^b \oplus \bbq^r \oplus \bbq^{2g_+} \oplus \bbq^{g_-} \] 
with the suggested basis.  More specifically, the generators consist of graph loops $\gamma_i$ (for $1\leq i \leq b$), surviving regular fibers $t^k$ (for $1\leq k \leq r$), orientable genus generators $\alpha_j^S$ and $\beta_j^S$ (for each orientable base $S$ and $1 \leq j \leq g(S)$), and nonorientable genus generators $\delta_i^S$ (for each nonorientable base $S$ and $1 \leq i \leq g(S)-1$).
\end{thm}

\subsection{Surviving regular fibers}\label{sec:SurvivingFibers}

Let $A$ denote the connectivity matrix as above.  Let $t^i$ denote the regular fibers, and think of them as forming a basis for the range of the associated linear transformation.  It is correct to think of this transformation as right multiplication by $A$, though since $A$ is symmetric, we may think of left multiplication.  Then $A$ is a `boundary map' in the sense that a row of $A$ describes the boundary of [a fraction of] an immersed surface in the graph manifold.  The surface corresponding to $t^i$ consists of the orientable base of that node together with the meridinal disk in the critical fiber and some annuli given by the node information (see Section~\ref{section:TSurface} for an explicit construction).  We must clear denominators in $A$ to get an actual, immersed surface.  This map $A$ records no information about the non-orientable regular fiber boundary components of these surfaces, but the $t^i$ from non-orientable ends die over $\bbq$.  We think of these immersed surfaces as a basis for the domain of this transformation.

The row-reduced echelon form of $A$ is convenient since it represents the same transformation with range-basis the same $t^i$ as before.  As discussed above, the survival or death information of the $t^i$ in $H_1(H;\bbq)$ is contained within $A$.  Therefore, the $r$ from Theorem~\ref{thm:homologyofgraphmanifold} is equal to the number of columns in $\mathrm{rref}(A)$ that do not contain a pivot position.  The $t^i$ in these free columns given the chosen basis for the $\bbq^r$ in Theorem~\ref{thm:homologyofgraphmanifold}.

Note that the surviving linear combinations of $t^i$ give rise to solutions to $xA = 0$.  From these solutions, we get linear combinations of the immersed surfaces mentioned above which form closed, immersed surfaces (perhaps after adding on some Klein bottles in non-orientable fibers; see Section~\ref{section:TSurface}).  These closed immersed surfaces can be used to construct intersection duals to the surviving $t^i$.  We illustrate this first with an example.

\begin{exam}\label{exam:SurvivingFiberFromTwoNodes}
	Take a Seifert fibered space with base $S^2$ and one critical fiber of type $-1/2$ and plumb it once (positively) to a Seifert fibered space with base $S^2$ and one critical fiber of type $-2$.  Let $S$ in the first Seifert fibered space denote the critical disk glued to one twice-punctured $S^2$ and two annuli connecting up to the other Seifert fibered space.  Then $\del S = 2t^1 + t^2$. In the other space, let $P$ denote the rational 2-chain with boundary $t^1 + \frac{1}{2}t^2$.  It is not representable by an immersed surface, but $2P$ is: take two copies of the base surface, the critical disc, and the annulus carrying a regular fiber near the critical fiber to a regular fiber on the boundary.  Here is a schematic of how to paste together these pieces to get an immersion of a closed surface:
	
	\begin{center}
		\begin{overpic}{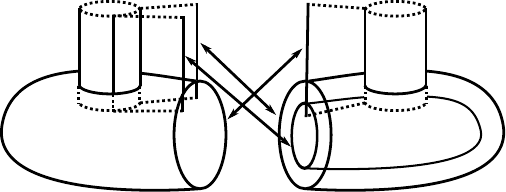}
			\put(20, 5){$S$}
			\put(80, 8){$2P$}
			\put(41, 32){$t^1$}
			\put(55, 32){$t^2$}
		\end{overpic}
	\end{center}
	
	The connectivity matrix is $A = \begin{pmatrix}
		2&1 \\ 1 & 1/2
	\end{pmatrix}$, which reduces to $\begin{pmatrix}
		2&1 \\ 0&0
	\end{pmatrix}$.  We declare $t^2$ to be our preferred generator for $H_1(M)$ and $t^1$ to be redundant. The kernel of $A$ is generated by $P - 1/2S$, so we clear denominators and set $F = 2P - S$.  The class $F$ is represented by a closed immersed surface: push all boundary components of $S$ and $P$ to meet at the plumbing, and glue the two $t^1$'s from the $S$ side to the two horizontal boundary components of $2P$ and the $t^2$ from the $2P$ side to the one horizontal boundary component of $S$.

	This $F$ -- rather, an embedded surface homologous to $F$ (described in \ref{section:TSurface} ) --  is the intersection dual of $t^2$.  In this simple example, there are no other homology classes to consider.  In Example~\ref{exam:ExWLoop}, we consider a different graph manifold that has more homology classes but the same connectivity matrix.
	
\end{exam}

\section{The intersection ring}\label{section:cohomology}

Instead of computing the cohomology ring $H^\ast(M)$ of a graph manifold directly, we will compute the intersection ring on $H_\ast(M)$.  It is well known that the cup product is Poincar\'e dual to intersection~\cite[p. 156]{HiltonWylie}. More precisely, if $A$ and $B$ are imbedded, smooth submanifolds of $M$ that intersect transversally, then $PD[A] \smile PD[B] = PD[A\cap B]$.  If $x$ is a curve representing a basis element of $H_1(M;\bbq)$, the intersection dual $ID(x)$ will denote the second homology class which can be represented by an embedded, oriented surface that intersects $x$ geometrically once and is disjoint from all other curves in the basis. The intersection product will be denoted $[A]\cdot[B] = [A\cap B]$.

By Theorem~\ref{thm:homologyofgraphmanifold}, there are five types of generators for $H_1(M)$ of a graph manifold:

\begin{enumerate}
	\item type $\alpha$: coming from the homology of the orientable base surfaces,
	\item type $\beta$: also coming from the orientable base surfaces, and dual on those surfaces to the $\alpha$ curves,
	\item type $\gamma$: coming from loops in the underlying graph,
	\item type $\delta$: coming from the homology of the nonorientable base surfaces, and
	\item type $t$: coming from surviving regular fibers.
\end{enumerate}

Note that representatives of $\alpha$, $\beta$, $\delta$, and $\gamma$ may be chosen with the following properties:

\begin{enumerate}
	\item (subsets of) the $\alpha$ and $\beta$ form symplectic bases of curves on each orientable base surface,
	\item the $\delta$ curves are disjoint,
	\item the $\gamma$ curves are comprised of arcs, each arc lying on a base surface and missing all $\alpha$, $\beta$, and $\delta$ curves and other $\gamma$ arcs.
\end{enumerate}


Our strategy in describing the ring $(H_\ast(M),\cdot)$ is to construct intersection duals to the above generators and then to define the intersection product $H_2(M)\times H_2(M)\to H_1(M)$ in terms of the dual basis. This will completely describe the ring structure, just as the cup product on $H^1(M)\times H^1(M)$ completely determines the full cohomology ring of a closed, oriented $3$-manifold.

\subsection{A basis for $H_2(M)$ via intersection duals}

\subsubsection{Duals of type $\alpha$ and $\beta$} Any type $\alpha$ generator may be represented by a circle $\alpha$ embedded on an orientable base surface $S$. This circle has a dual circle, $\beta$, on $S$, and $\beta$ is a type $\beta$ generator of $H_1(M)$.  Let $t$ denote the regular fiber over this base surface ($t$ is not necessarily a generator of $H_1(M)$).  The torus $\beta \times t$ intersects $\alpha$ in a point. Since $\beta \times t$ is supported in a Seifert fibered end of $M$, $\alpha$ and $\beta$ are dual on $S$, and the $\gamma$ curves are disjoint from the $\alpha$'s and $\beta$'s, we see that $[\beta\times t] = ID[\alpha]$.  Likewise, $-[\alpha\times t]=ID[\beta]$.

\subsubsection{Duals of type $\delta$}\label{sec:klein}  $H_1( \#_n \bbr P^2)\cong \bbq^{n-1}$ with a basis consisting of disjoint curves each with self-intersection one (i.e. their regular neighborhoods are M\"obius bands). To see this, take the usual $2n$-gon presentation for $\#_n \bbr P^2$ and pick the generators to be pushoffs of $n-1$ of the sides (with distinct labels). In the orientable $S^1$ bundle over $\#_n \bbr P^2$, there are $n-1$ disjoint Klein bottles over these generator circles.  The Klein bottles are incarnations of the relations $\delta\,t\delta\inv\,t=1$ in the fundamental group. Each of these Klein bottles intersects the corresponding $\delta$ generator algebraically once and is $ID[\delta]$.

\subsubsection{Duals of type $\gamma$} The type $\gamma$ curves generate the homology of $M$ that comes from the topology of the underlying graph. Let $T$ denote the maximal tree of the graph of $M$, as in Section~\ref{section:homology}. For each edge $e$ not in $T$, there is a plumbing (and a plumbing torus $\tau$) corresponding to $e$.  One may pick $b$ disjoint arcs in $mT$ so that after performing these plumbings, the arcs close up and become the full set of $\gamma$ curves (there are $b$ of them). Each plumbing torus $\tau$ is $ID[\gamma]$, since $\tau = t_1 \times t_2$ is a product of regular fibers from two adjacent Seifert fibered pieces in $M$.

\subsubsection{Duals of type $t$}\label{section:TSurface} The duals of the type $t$ curves are the most difficult to construct, and we have already introduced these duals in Section~\ref{sec:SurvivingFibers} from the upside-down point of view. Recall that the connectivity matrix can be thought of a boundary map that counts the boundary of immersed surfaces in the graph manifold (given by orientable bases with traces of regular fibers across those bases). In the discussion of Section~\ref{sec:SurvivingFibers}, the type $t$ generators for $H_1(M)$ correspond to the kernel of the connectivity matrix. If $F$ is a nontrivial element in the kernel of the connectivity matrix, then some integer multiple of $F$ can be built by a finite number of copies of the immersed surfaces with boundary, like $S$ and $P$ from Example~\ref{exam:SurvivingFiberFromTwoNodes}. The result is an immersed surface that may be closed but may have boundary consisting of some fibers from nonorientable base Seifert fibered ends, but these curves have order two in the nonorientable ends and may be capped off (after perhaps taking two copies of this surface) with punctured Klein bottles inside these ends. These Klein bottles are as in \ref{sec:klein} using the type $\delta$ curve \emph{not} used in the basis for $H_1$ (this helps ensure certain surfaces are disjoint). The result is a closed immersed surface $F$ with some type $t$ curve that intersects $F$ in a finite number of points. We may choose a basis for the type $t$ curves so that each $t$ curve intersects one and only one of these immersed surfaces.

To recap, $F$ is built up from various base surfaces throughout the graph manifold, annuli connecting up regular fibers of the various ends, critical discs, and Klein bottles.  There are choices, especially about how these annuli move around the graph to connect various regular fibers.  These choices are dictated by the graph, though not uniquely; this ambiguity is accounted for in the bottom right entry of the matrix in Theorem~\ref{thm:IntersectionRing}.  Note that no annuli need to approach an unoriented base end.

The immersed surface $F$ may be distorted to an embedded one. Its self-intersections consist of arcs where the annuli and base spaces intersect, and these arcs proceed from the neighborhood of a critical fiber to the boundary of the end. These may be resolved as follows (note the orientations): 

\begin{center}
	\begin{overpic}{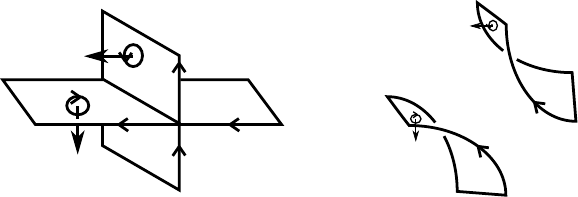}
		\put(53, 20){$\leadsto$}
	\end{overpic}
\end{center}

Note that this resolution is compatible with the critical disks, which are embedded.  Let us take for example a torus around a critical fiber of type $-2/1$, to mirror what is happening with the $2P$ in Example~\ref{exam:SurvivingFiberFromTwoNodes}.  Two copies of the base surface are needed to match up with one annulus. On the torus near the critical fiber, the resolution is 

\begin{center}
	\begin{overpic}{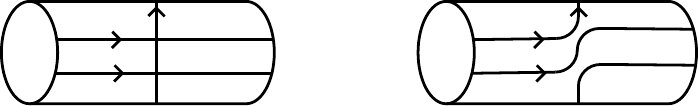}
		\put(45, 5){$\leadsto$}
	\end{overpic}
\end{center}

which can be capped off, using the critical disc, to form a closed, embedded surface homologous to $F$.

The resulting embedded surface (possibly scaled by a constant to ensure $[F]\cap [t]=1$) may be taken to be $ID[t]$. 

\subsection{The intersection product on $H_2$}

A symplectic pair of generators from end $i$ of type $\alpha$ and $\beta$ have duals that intersect each other according to $ID[\beta]\cap ID[\alpha]=[\alpha\times t]\cap[\beta\times t]=[t]$, and these duals may intersect the duals of type $t$ according to $ID[\beta]\cap ID[t]=[\alpha\times t]\cap[F]=n_i[\beta]$, etc., and are otherwise disjoint from all duals. The duals of type $\delta$ are disjoint Klein bottles, and they are disjoint from all other duals since they are supported in the ends over nonorientable bases. A dual of type $\gamma$ is a plumbing torus and may intersect the duals of type $t$; these intersections will be copies of type $t$ curves, as in $ID[\gamma]\cap ID[t]=[t_i\times t_j]\cap [F]=n_i[t_j]+n_j[t_i]$. Otherwise, the duals to type $\gamma$ will be disjoint from all other duals. The embedded surfaces $F$ that are dual to generators of type $t$ have trivial self-intersection, but they may have intersections among themselves as $[F]\cap [F']=\sum C_i[\gamma_i]$ for some $C_i$.

We summarize this intersection product in the following theorem:

\begin{thm}\label{thm:IntersectionRing}
	Let $M$ be a closed, orientable graph manifold. Let $A,B,D,C$ and $T$ denote the subspaces of $H_2(M;\bbq)$ spanned by the intersection duals to the $H_1$ generators of type $\alpha, \beta,\delta,\gamma$ and $t$, respectively. Then the following table summarizes the intersection product on $H_2(M;\bbq)$:
	
	\begin{center}
	\begin{tabular}{c|ccccc}
		$\cdot$& $A$ & $B$ & $D$ & $C$ & $T$\\ \hline
		$A$& $0$ & $t$s & $0$ & $0$& $\beta$s \\
		$B$& $t$s & $0$ & $0$ & $0$ & $\alpha$s \\
		$D$& $0$ & $0$ & $0$ & $0$ & $0$ \\
		$C$& $0$ & $0$ & $0$ & $0$ & $t$s \\
		$T$& $\beta$s & $\alpha$s & $0$ & $t$s & $\gamma$s
	\end{tabular}
	\end{center}
	An entry of this table may be read as, for example, ``The duals to the $\alpha$ curves intersect the duals to the $t$ curves homologically in linear combinations of the $\beta$ curves.''  As is generally the case, this intersection product is skew-symmetric.  Moreover, for each type $\alpha$ generator, there is either zero or one type $\beta$ generator whose dual intersects $\alpha$'s dual nontrivially.
\end{thm}

\begin{exam}\label{exam:ExWLoop}
	Let $M$ be the graph manifold with three nodes, all plumbed by $+J$ by a triangular graph, such that the bases for the Seifert fibered ends are

	\parbox{.7\textwidth}{	\begin{itemize}
			\item $P$: genus $-3$ with one critical fiber of type $1/1$
			\item $Q$: genus $1$ with one critical fiber of type $-1/2$
			\item $R$: genus $1$ with one critical fiber of type $-2/1$
		\end{itemize}
		}\parbox{.2\textwidth}{
		\begin{center}
			\begin{overpic}{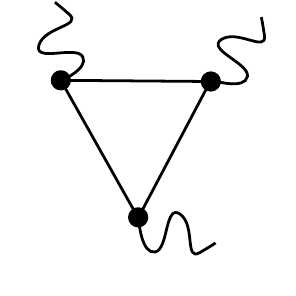}
				\put(0, 70){$P$}
				\put(42, 76){$+$}
				\put(85, 70){$Q$}
				\put(22, 45){$+$}
				\put(60, 45){$+$}
				\put(30, 20){$R$}
			\end{overpic}
		\end{center}
		}
	
	Let $t_P$, $t_Q$, and $t_R$ denote the regular fibers corresponding to the ends with bases $P,Q$ and $R$, respectively. Let $\alpha_1,\beta_1$ be a symplectic basis of curves on $Q$ and $\alpha_2,\beta_2$ be a symplectic basis of curves on $R$. Let $\delta_1$ and $\delta_2$ be disjoint type $\delta$ generators on $P$. There is one graph loop $\gamma$, whose dual is chosen to be the plumbing torus between $P$ and $R$.

	Recall that the connectivity matrix does not record information about nonorientable vertices, so this matrix is the same as in Example~\ref{exam:SurvivingFiberFromTwoNodes}. Note an immersed surface may be constructed from $R-2Q$ but has two boundary components, each of which is $t_P$. Since $2t_p=0$ inside the Seifert fibered space over $P$, we can cap off $R-2Q$ inside that end and get an immersed surface $F$.  Note that $F$ is dual to $t_R$ since $F$ contains one copy of $R$. We choose $t_R$ as the generator of the regular fiber subspace and note that $t_q = -2t_R$.
	
	A basis for $H_1(M;\bbq)$ is $\{\alpha_1,\beta_1,\alpha_2,\beta_2,\delta_1,\delta_2, \gamma, t_R\}$ with dual basis $\{A_1,B_2,A_2,B_2,D_1,D_2,C,F\}$, where all but $C$ and $F$ are described in 4.1.1 and 4.1.2. The dual elements $C$ and $F$ are described in the previous two paragraphs.
	
	Geometrically, $A_1\cap B_1 = t_Q$ and $A_2\cap B_2 = t_R$; other than the intersections with $F$, these are the only intersections of any $A$ or $B$ with any other elements in the basis. Since $F$ contains one positive copy of $R$, we see $A_2\cap F = -\beta_2$ (since $A_2\cap B_2 = t_R$ and $F$ is the dual of $t_R$).  Similarly, $B_2\cap F = \alpha_2$, $B_1\cap F = -2\alpha_1$, and $A_1\cap F = 2\beta_1$.  Geometrically, the duals $D_i$ represented by Klein bottles are disjoint from all other members of the basis, and $C\cap F = t_p$, which is trivial in homology.  The intersection form on homology is
	
	\begin{center}
		\begin{tabular}{c|cccccccc}
			$\cdot$ & $A_1$ & $B_1$ & $A_2$ & $B_2$ & $D_1$ & $D_2$ & $C$ & $F$\\ \hline
			$A_1$ & $0$ & $-2t_R$ & $0$ & $0$ & $0$ & $0$ & $0$ & $2\beta_1$ \\
			$B_1$ & $2t_R$ & $0$ & $0$ & $0$ & $0$ & $0$ & $0$ & $-2\alpha_1$ \\
			$A_2$ & $0$ & $0$ & $0$ & $t_R$ & $0$ & $0$ & $0$ & $-\beta_2$ \\
			$B_2$ & $0$ & $0$ & $-t_R$ & $0$ & $0$ & $0$ & $0$ & $\alpha_2$ \\
			$D_1$ &$0$ & $0$ & $0$ & $0$ & $0$ & $0$ & $0$ & $0$ \\
			$D_2$ & $0$ & $0$ & $0$ & $0$ & $0$ & $0$ & $0$ & $0$ \\
			$C$ & $0$ & $0$ & $0$ & $0$ & $0$ & $0$ & $0$ & $0$ \\
			$F$ & $-2\beta_1$ & $2\alpha_1$ & $\beta_2$ & $-\alpha_2$ & $0$ &  $0$ & $0$ & $0$ 
		\end{tabular}
	\end{center}
	
	Note that $D_1,D_2,$ and $C$ do not intersect anything homologically and split off a $\bbq^3$ summand of $H_2$ with trivial intersection form.  The remaining $A_1,B_1,A_2,B_2,F$ intersect in a way reminiscent of the genus two surface cross the circle; the only difference is that the $A_1$ and $B_1$ intersection is weighted differently than the intersection between $A_2$ and $B_2$.  This phenomenon is explained rigorously in Section~\ref{section:TreeManifolds}.
	
\end{exam}

\section{Examples: tree graph manifolds}\label{section:TreeManifolds}

Define a \emph{tree graph manifold} to be a graph manifold whose associated graph is a tree. Throughout this section, let $M$ be a tree graph manifold where each Seifert fibered space in the manifold has an orientable base.  Thus there are no type $\delta$ and no type $\gamma$ generators for $H_1(M)$, only type $\alpha$, $\beta$, and $t$. It is convenient to think of the $\alpha$s and $\beta$s as living on orientable surfaces $S_i$ (with regular fibers $t_i$). However, some of the $t_i$s may be redundant. Recall that we chose a basis for the type $t$ subspace of $H_1(M)$ by first constructing its dual basis from linear combinations of the $S_i$. This is a subtle matter that will be clarified through definitions and examples.

\begin{defn}\label{def:ConnectedSum}(cf.~\cite{AAM:ConnectedSumsOfRings})
	Let $R$ and $S$ be cohomology rings of closed, oriented $3$-manifolds (more generally, local Gorenstein rings), and let $T$ be $\bbq^n$ for some $n$ (more generally, a local ring). Given homomorphisms $\epsilon_R: R \longrightarrow T \longleftarrow S: \epsilon_S$, form the fiber product $R \times_T S := \{ (r,s): \ \epsilon_R(r)=\epsilon_S(s) \}$. Given a $T$-module $V$ and homomorphisms $\iota_R:V\to R$ and $\iota_S:V\to S$ that satisfy $\epsilon_R\iota_R = \epsilon_S\iota_S$, form the \emph{connected sum of rings} by \[ R\#_T S := \left( R\times_T S \right) / \left\{ \left( \iota_R(v), \iota_S(v) \right): v\in V\right\} \]
\end{defn}

The connected sum of rings has as a special case the ring of a connected sum of manifolds. Let $M$ and $N$ be $3$-manifolds with cohomology rings $R = R^0 \oplus R^1 \oplus R^2 \oplus R^3$ and $S = S^0 \oplus S^1 \oplus S^2 \oplus S^3$, respectively. We will work with $\bbq$ coefficeints. The cohomology ring of $M\# N$, morally speaking, has the $H^1$ and $H^2$ of $M$ and of $N$, but has only one $\bbq$ for $H^0$ and one $\bbq$ for $H^3$.  Let $T = \bbq$ and take $\epsilon_R$ and $\epsilon_S$ to be the projections onto the degree zero pieces of those rings. These are ring homomorphisms onto $\bbq$. Some examples of elements in $R\times_T S$ are $(1,1)$ for $1\in R^0=\bbq$ and $1\in S^0=\bbq$, $(r,0)$ for $r\in R^1\oplus R^2\oplus R^3$, and $(0,s)$ for $s\in S^1\oplus S^2 \oplus S^3$. In other words, the pushout $R\times_T S$ sets the degree $0$ part of $R$ and $S$ equal while retaining all elements of higher degree (and their products). To get $H^\ast(M\# N)$, one must set the fundamental class of $M$ equal to the fundamental class of $N$, and this is achieved by $\iota_R$ and $\iota_S$. Define $V=\bbq$ and $\iota_R$ and $\iota_S$ by $\iota_R(1) = 1 = [M] \in R^3 = \bbq$ and $\iota_S(1) = 1 = [N] \in S^3 = \bbq$. Then $\epsilon_R\iota_R=0=\epsilon_S\iota_S$ and so $R\#_T S$ is defined. It is easy to check that $R\#_T S = H^\ast(M\# N)$.

In all of our examples, $V$ will be $\bbq$, $T$ will be $\bbq$, $\iota_R$ and $\iota_S$ will be as above, and $\epsilon_R$ and $\epsilon_S$ will be defined on $R^0\oplus R^1$ and $S^0 \oplus S^1$.  In all cases $\epsilon_R\iota_R=0=\epsilon_S\iota_S$.

In a tree graph manifold, there are tori $\alpha_{i,j} \times t_j$ and $\beta_{k,l} \times t_l$ that live over different bases $\Sigma_j$ and $\Sigma_l$. Of course $(\alpha_{i,j}\times t_j) \cap (\beta_{k,l} \times t_l) = \delta_{ik}\delta_{jl}t_j$ \emph{geometrically}, but to write this intersection down algebraically, one has to know how to write $t_j$ in terms of our chosen basis for $H_1=H^2$.  The idea to compute the intersection product is to calculate it on the Seifert fibered pieces and then patch it all together via the connected sum (of rings) construction.

\begin{thm}
	The cohomology ring of a tree graph manifold with orientable bases only is a connected sum of the cohomology rings of $\widehat{\Sigma}_i \times S^1$, where the $\Sigma_i$ are the orientable bases in the tree graph manifold and where $\widehat{\ }$ denotes closing the surface by capping off boundary components.
\end{thm}

The proof is a generalization of the following example.

\begin{exam}\label{ex:ConSum}
	Let $M$ be the tree graph manifold with three Seifert fibered spaces over the genus one surface, each with one critical fiber; the critical fibers will be of type $-1/2$, $-3$, and $-1$. Both plumbings are by $+J$. The graph manifold is depicted schematically below.
	
	\parbox{.7\textwidth}{	\begin{itemize}
			\item $P$: genus $1$ with one critical fiber of type $-1/2$
			\item $Q$: genus $1$ with one critical fiber of type $-3/1$
			\item $R$: genus $1$ with one critical fiber of type $-1/1$
		\end{itemize}
		}\parbox{.2\textwidth}{
		\begin{center}
			\begin{overpic}{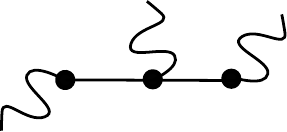}
				\put(85, 4){$R$}
				\put(50, 4){$Q$}
				\put(26, 22){$P$}
			\end{overpic}
		\end{center}
		}
	
	We will construct the intersection ring from the geometric point of view and then give the connected sum description of that ring.
	
	The connectivity matrix is \[ \bordermatrix{ &t_P& t_Q& t_R\cr
	                P & 1/2 & 1  & 0 \cr
	                Q & 1  &  3 & 1 \cr
	                R & 0  &   1 & 1  } \]
	and reduces to 				\[ \bordermatrix{ &t_P& t_Q& t_R\cr
					                P & 1/2 & 1  & 0 \cr
					                R+2P-Q & 0  &  0 & 0 \cr
					                R & 0  &   1 & 1  } \]
	
	We choose $t=t_R$ to generate the regular fiber subspace of $H_1$ and see $t_P = 2t$ and $t_Q = -t$. A closed immersed surface $S$ can be built from $R+2P-Q$, and $S$ is the intersection dual of $t$. There are curves $\alpha_P$ and $\beta_P$ on $P$ that are dual on $P$, and there are curves $\alpha_Q$, $\beta_Q$, $\alpha_R$ and $\beta_R$ with similar properties.  These curves have dual tori $A_P,B_P,A_Q,B_Q,A_R$ and $B_R$. Note that $A_P\cap B_P=t_P$ geometrically but $[A_P]\cdot[B_P] = 2t$ homologically.  The intersection $[A_Q]\cdot[B_Q]$ is $-t$, and $[A_R]\cdot[B_R] = t$. The class $[S]$ intersects the dual tori in the predictable ways.  Let $M_P$ denote the fundamental class of the $T^2 \times S^1$ labeled $P$.
	Let us construct this intersection ring algebraically. Let $H(P)$ denote the intersection ring of $T^2 \times S^1$ with generating curves $\alpha_P, \beta_P$, and $t_P$; $H(Q)$ and $H(R)$ are defined similarly.  In these rings, $P$, $Q$, and $R$ denote the tori.  Intuitively, the intersection ring of $M$ is obtained by $H(P)$, $H(Q)$, and $H(R)$ by setting $t_P$ and $t_Q$ equal to certain multiples of $t_R$, as we saw in the preceding paragraph.  Let $T$ be the truncated polynomial ring $T = \bbq\left[ F \right]/\langle F^2 \rangle$.  To construct the connected sum, we need to define ring homomorphisms $\epsilon_P:H(P) \to T$ and $\epsilon_Q$ and $\epsilon_R$.  We should think of element $F \in T$ as the surface $S$ in the geometric construction of $H_\ast(M)$.  Define $\epsilon_P$ by taking $1\mapsto 1$ and $P\mapsto \frac{1}{2}F$ and all other generators of $H(P)$ to zero; this is a ring homomorphism since $P$ is not expressible as a nontrivial product of elements in $H(P)$.  Define $\epsilon_Q$ by taking $1\mapsto 1$ and $Q\mapsto -F$ and others to zero. Define $\epsilon_R$ by taking $1\mapsto 1$ and $R\mapsto F$ and others to zero.
	The fiber product $P\times_T Q \times_T R$ consists of those triples in $P\times Q\times R$ with whose entries get sent to the same element of $T$ under the $\epsilon$ maps.  A generating set for this fiber product is $\{(1_P,1_Q,1_R),(A_P,0,0),(B_P,0,0),(0,A_Q,0),(0,B_Q,0),(0,0,A_R),(0,0,B_R), (2P,-Q,R), \\(\alpha_P,0,0),(\beta_P,0,0),(0,\alpha_Q,0),(0,\beta_Q,0), (0,0,\alpha_R),(0,0,\beta_R), (t_P,0,0),(0,t_Q,0),(0,0,t_R),\\(M_P,0,0),(0,M_Q,0), (0,0,M_R) \}$. At this point in the construction, we have built a ring with unit, the right number of `surfaces', but too many `curves' and `fundamental classes'. We must take the quotient that sets the fundamental classes equal and sets the right multiples of the $t$s equal.  Quotient $P\times_T Q\times_T R$ by the relations $(t_P,0,0) = (0,0,2t_R)$, $(0,t_Q,0) = (0,0,-t_R)$, $(M_P,0,0) = (0,M_Q,0)$, and $(0,M_Q,0) = (0,0,M_R)$.  The resulting quotient is the connected sum $P\#_T Q \#_T R$ and is isomorphic to $H_\ast(M)$; the isomorphism we have in mind sends, in particular, $(0,0,t_R)$ to $t$ and $(2P,-Q,R)$ to $S$, where $t$ and $S$ are defined in the geometric construction.

\end{exam}

\section{Applications: a manifold which is not homology cobordant to a tree graph manifold}\label{section:ExampleAwesome}

Not every manifold is homology cobordant to a tree graph manifold, as seen. 

\begin{thm}\label{thm:TreeGraphManifoldsNotGeneric}
	Not every closed, compact $3$-manifold is homology cobordant to a tree graph manifold.
\end{thm}

The proof relies on the notion of a decomposition of the cohomology ring of a $3$-manifold.  See Definition~\ref{def:ConnectedSum} for the notion of composition.

\begin{defn}
	Let $U$ be the cohomology ring of a closed, orientable $3$-manifold. A \emph{decomposition} of $U$ corresponds to seeing $U$ as the cohomology ring of a connected sum of manifolds, i.e. $U = H^\ast(M\# N)$.\footnote{For our purposes, it is sufficient to consider the ring decomposition corresponding to the connected sum decomposition of $3$-manifolds. One could consider more general decompositions as suggested by Definition~\ref{def:ConnectedSum}, but the simpler notion is strong enough for us.} The ring $U$ is \emph{indecomposable} if $U=H^\ast(M\# N)$ implies that $M$ or $N$ is a homology $3$-sphere. If $U=H^*(M\# N)$ and $\rank\left(H_2(M) \right) = r$, we say \emph{$U$ splits a rank $r$ summand}.
\end{defn}

For example, if $U$ splits a rank $1$ summand, then $U = H^\ast\left( S^1\times S^2 \# N\right)$ for some $3$-manifold $N$. Note that the total rank of $H^\ast\left( S^1 \times S^2\right)$ is $4$, but the rank of $H_2(S^1 \times S^2)$ is $1$.  Note that \emph{summand} means (just the $H_2$ part of the) \emph{connected summand}, not \emph{direct summand}.

	
	

Before presenting an example of a 3-manifold that is not homology cobordant to any tree graph manifold, we recall a result of D. Sullivan \cite{DSullivan:IntersectionRing}: for any finitely generated free abelian group $H$ and any $\omega \in \bigwedge^3 H$, there is a closed, compact 3-manifold $M$ whose cohomology ring corresponds to $\omega$.  We briefly recount the correspondence between $\bigwedge^3 H$ and the cohomology rings of 3-manifolds with $H_2 \cong H$. Let $M$ be a closed, compact 3-manifold. Any three surfaces $A,B,C$ in $M$ will intersect generically in finitely many points; this `triple intersection' is skew symmetric. Given a basis for $H \cong H_2(M)$ and any three elements $\alpha, \beta, \gamma$ in that basis, represent the elements by embedded surfaces $A$, $B$, and $C$. Suppose there are $m$ (signed) intersection points between $A$, $B$, and $C$. The three-form corresponding to $H^*(M)$ will have $m\,\alpha\wedge\beta\wedge\gamma$.  As a quick example, zero-surgery on the Borromean rings (or $T^3$) has the form $\alpha\wedge\beta\wedge\gamma$, and $\Sigma_2 \times S^1$ has the form $\alpha_1\wedge\beta_1\wedge t + \alpha_2\wedge\beta_2\wedge t$. Note that one can glean the intersection product from this $3$-form. In $\Sigma_2 \times S^1$, we see $[\alpha_1]\cdot[\beta_1] = ID[t]$, and in $T^3$, we have $[\alpha]\cdot[\gamma] = -ID[\beta]$.

\begin{exam}\label{ex:Awesome}
	Let $M$ denote a 3-manifold with $H_2(M) \cong \mathbb{Q}^6$ with fixed basis $\{a,b,c,d,e,f\}$ and cohomology ring corresponding to the 3-form $abc + aef + bde$.  Such an $M$ exists by work of D. Sullivan \cite{DSullivan:IntersectionRing}, and an explicit example is zero-surgery on the link in Figure~\ref{fig:ExampleAwesome}.
	
	\begin{figure}[!ht]
		\begin{center}
			\begin{overpic}{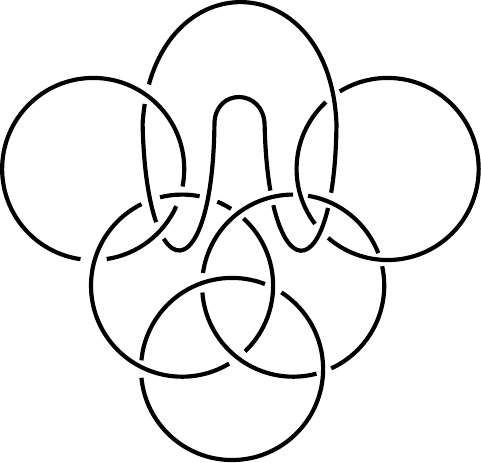}
				\put(18,20){$a$}
				\put(-5, 50){$f$}
				\put(48, 78){$e$}
				\put(100, 50){$d$}
				\put(75, 18){$b$}
				\put(48, 4){$c$}
			\end{overpic}
			\caption{$M$ is zero-surgery on this link and has the $3$-form $abc + aef + bde$}
			\label{fig:ExampleAwesome}
		\end{center}
	\end{figure}
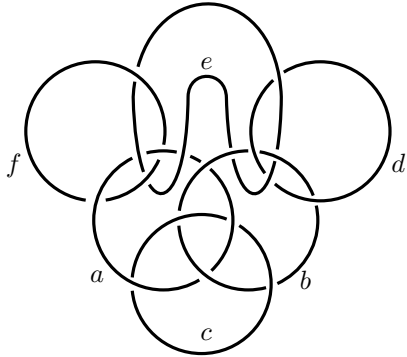
	
	We claim that $H^\ast(M)$ splits off neither a $\bbq$ nor a $\bbq^3$ summand, i.e., up to cohomology $M$ does not split off an $S^1 \times S^2$ or an $S^1 \times S^1 \times S^1$.
	
	If $H^\ast(M)$ split a $\bbq$ summand, there would be a nonzero $x\in H_2(M)$ with $x\cdot y=0$ for all $y\in H_2(M)$.  A lower case $y$ denotes an element of $H_2(M)$, and a capital $Y$ denotes its intersection dual in $H_1(M)$.  Suppose $x\cdot y=0$ for all $y\in H_2(M)$. Write $x = n_1\,a+ n_2\,b + n_3\,c + n_4\,d + n_5\,e + n_6\,f$. Taking $y=c$, we compute $x\cdot c = -n_1\,B + n_2\,A$ and conclude $n_1=n_2=0$. Taking $y=f$, we compute $x\cdot f = n_5\,A$ and conclude $n_5=0$. Taking $y=e$ we compute $x \cdot e = -n_6\,A+ n_4\,B$ and conclude $n_4=n_6=0$. Taking $y=b$, we compute $x \cdot b = -n_3\,A$ and conclude $n_3=0$.
	
	Thus $x=0$, and $H^\ast(M)$ splits no $\bbq$ summand.
	
	That $H^\ast(M)$ splits no $\bbq^3$ summand is left to Theorem~\ref{thm:splitsnoz3}.

\end{exam}

\begin{thm}\label{thm:splitsnoz3}
	The form $\omega = abc + aef + bde$ is not equivalent, up to a change of basis of $H_2(M)$, to the form $\omega' = uvw + xyz$.
\end{thm}
\begin{proof}
	The form $\omega$ describes the multiplication $H_2 \times H_2 \to H_1$ in the graded cohomology ring $H^\ast(M)$. An invertible map $N$ that changes the basis of $H_2(M)$ acts on the form $\omega \mapsto \omega^N$. Suppose for the sake of contradiction that there is an element $N\in GL(\bbq^6)$ taking the basis $\{a,b,c,d,e,f\}$ to the basis $\{u,v,w,x,y,z\}$ for $H_2(M)\cong \bbq^6$ so that $\omega^N = uvw + xyz$. The form $\omega^N$ describes the cohomology ring of $S^1\times S^1 \times S^1 \# S^1\times S^1 \times S^1$; this is the only ring of a $3$-manifold that splits a $\bbq^3$ summand but not a $\bbq$ summand.
	
	Let $\{A,B,C,D,E,F\}$ be the dual (to $\{a,b,c,d,e,f\}$) basis for $H_1(M)$. From the form $\omega$, we can write the multiplication table:
	
	\[\omega \equiv \begin{tabular}{c|c|c|c|c|c|c|c|c|c|c|c|c|c|c}
 		$a\cdot b$ & $a\cdot c$ & $a\cdot d$ & $a\cdot e$ & $a\cdot f$ & $b\cdot c$ & $b\cdot d$ & $b\cdot e$ & $b\cdot f$ & $c\cdot d$ & $c\cdot e$ & $c\cdot f$ & $d\cdot e$ & $d\cdot f$ & $e\cdot f$ \\ \hline
		 $C$ & $-B$ & $0$ & $F$ & $-E$ & $A$ & $E$ & $-D$ & $0$ & $0$ &  $0$ & $0$ & $B$ & $0$ & $A$
	\end{tabular}\]
	
	Write 
	\begin{eqnarray*}
		u &=& n^1_1\,a + n^1_2\,b + n^1_3\,c + n^1_4\,d + n^1_5\,e + n^1_6\,f \\
		v &=& \cdots\\
		w &=& \cdots\\
		x &=& \cdots\\
		y &=& \cdots\\
		z &=& n^6_1\,a + n^6_2\,b + n^6_3\,c + n^6_4\,d + n^6_5\,e + n^6_6\,f \\
	\end{eqnarray*}
	
	Let $\{U,V,W,X,Y,Z\}$ be the dual (to $\{x,y,z,x,y,z\}$) basis for $H_1(M)$. By hypothesis, the form $\omega^N$ gives the multiplication table:
	\[\omega^N \equiv \begin{tabular}{c|c|c|c|c|c|c|c|c|c|c|c|c|c|c}
 		$u\cdot v$ & $u\cdot w$ & $u\cdot x$ & $u\cdot y$ & $u\cdot z$ & $v\cdot w$ & $v\cdot x$ & $v\cdot y$ & $v\cdot z$ & $w\cdot x$ & $w\cdot y$ & $w\cdot z$ & $x\cdot y$ & $x\cdot z$ & $y\cdot z$ \\ \hline
		 $W$ & $-V$ & $0$ & $0$ & $0$ & $U$ & $0$ & $0$ & $0$ & $0$ &  $0$ & $0$ & $Z$ & $-Y$ & $X$
	\end{tabular}\]
	
	Using $N$ and the table for $\omega$, one computes
	\begin{eqnarray}\label{eq:product}
		u\cdot x = \left( n^1_2n^4_3 - n^1_3n^4_2 + n^1_5n^4_6 - n^1_6n^4_5 \right)\,A &+& \left( - n^1_1n^4_3 + n^1_3n^4_1 + n^1_4n^4_5 - n^1_5n^4_4 \right)\,B \\ + \left( n^1_1n^4_2 - n^1_2n^4_1 \right)\,C  &+& \left( - n^1_2n^4_5 + n^1_5n^4_2 \right)\,D \notag \\ + \left( - n^1_1n^4_6 + n^1_6n^4_1  + n^1_2n^4_4 - n^1_4n^4_2 \right)\,E &+& \left( n^1_1n^4_5 - n^1_5n^4_1 \right)\,F \notag 
	\end{eqnarray}
	
	One may use equation~(\ref{eq:product}) to calculate the other products; for example $v\cdot z$ may be gotten from equation~(\ref{eq:product}) by changing the superscripts to $2$ and $6$ instead of $1$ and $4$.
	
	The goal is to show that the tables for $\omega$ and $\omega^N$ are irreconcilable.
	
	\begin{prop}\label{prop:OppPara}
		For $i\in\{1,2,3\}$ and $j\in\{4,5,6\}$, we have parallel vectors of $\bbq^2$: \[ \begin{pmatrix} n^i_1 \\ n^i_2 \end{pmatrix} \Big\| \begin{pmatrix} n^j_1 \\ n^j_2 \end{pmatrix},\ \ \begin{pmatrix} n^i_2 \\ n^i_5 \end{pmatrix} \Big\| \begin{pmatrix} n^j_2 \\ n^j_5 \end{pmatrix},\ \ \begin{pmatrix} n^i_1 \\ n^i_5 \end{pmatrix} \Big\| \begin{pmatrix} n^j_1 \\ n^j_5 \end{pmatrix}\]
	\end{prop}
	\begin{proof}
		Equation~(\ref{eq:product}) gives an expansion for the product $u\cdot x$. By the multiplication table for $\omega^N$, $u\cdot x = 0$. Thus, the coefficients in front of $A,B,C,D,E$, and $F$ must be zero.  In particular, the coefficients in front of $C,D$, and $F$ are determinants of certain $2\times 2$ matrices. This concludes the proof for $(i,j)=(1,4)$. The other cases follow from similar arguments since $u\cdot y = u\cdot z = v\cdot x = v\cdot y = v\cdot z = w\cdot x = w\cdot y = w\cdot z = 0$.
	\end{proof}
	
	\begin{prop}\label{prop:AllPara}
		There is a (non-ordered) pair $(p,q)\in\{(1,2),(2,5),(1,5)\}$ such that \[ \begin{pmatrix} n^i_p \\ n^i_q \end{pmatrix} \Big\| \begin{pmatrix} n^j_p \\ n^j_q \end{pmatrix}\] for all $i,j\in\{1,2,3,4,5,6\}$.
	\end{prop}
	\begin{proof}
		We start by showing there is an $\ell\in\{1,2,3\}$, an $m\in\{4,5,6\}$, and a $k_\ell$ and $k_m\in\{1,2,5\}$ such that \[ n_{k_\ell}^\ell \neq 0 \neq n_{k_m}^m \]
		
		Suppose that $n_{k_\ell}^\ell=0$ for all $\ell\in\{1,2,3\}$ and $k_\ell\in\{1,2,5\}$.  We have assumed $u\cdot v = W$, but expanding $u\cdot v$ as in equation~(\ref{eq:product}) yields $u\cdot v=0$ (note that each product of $n$'s in equation~(\ref{eq:product}) has a $n_1$ or $n_2$ or $n_5$). By contradiction, there is an $\ell$ and a $k_\ell$ as required. A similar argument will produce an $m$ and a $k_m$ as desired.
		
		The reader will recall that parallelism is not an equivalence relation on the set of vectors in $\bbq^2$. However, one does have that $\vec{\alpha}\|\vec{\beta}$ and $\vec{\alpha}\|\vec{\gamma}$ implies $\vec{\beta}\|\vec{\gamma}$ as long as $\vec{\alpha}\neq \vec{0}$.  We will call this `weak transitivity of parallelism.'
		
		Set $p=k_\ell$ and pick $q=k_m$ unless $k_m=k_\ell$ in which case choose a $q\in\{1,2,5\}\setminus\{p\}$.  In either case we have \[ \begin{pmatrix} n^\ell_p \\ n^\ell_q \end{pmatrix} \neq \vec{0}\neq  \begin{pmatrix} n^m_p \\ n^m_q \end{pmatrix}\] Now we may use weak transitivity of parallelism with Proposition~\ref{prop:OppPara} to conclude the proof of the proposition.
	\end{proof}
	
	Returning to the proof of Theorem~\ref{thm:splitsnoz3}, we have by Proposition~\ref{prop:AllPara} a $(p,q)\in\{(1,2),(2,5),(1,5)\}$ with $n^i_pn^j_q - n^i_qn^j_p=0$ for all $i,j\in\{1,2,3,4,5,6\}$. By the expansion in equation~(\ref{eq:product}), there exists one of $B,C$, and $F$ that never appears in any of the products among the $u,v,w,x,y,z$. Since the duals $U,V,W,X,Y,Z$ are all among these products, we conclude that $\mathrm{span}\{U,V,W,X,Y,Z\} \subsetneq \mathrm{span}\{A,B,C,D,E,F\}$ which contradicts that $\{U,V,W,X,Y,Z\}$ is a basis for $H_1(M)\cong \bbq^6$. This contradicts the existence of $N$ and concludes the proof of Theorem~\ref{thm:splitsnoz3}.
	
\end{proof}

\begin{thm}\label{thm:TreeSplitsOddly}
	Let $M$ be a tree graph manifold with $H_1(M)$ of rank 6. Then $H^*(M)$ has an indecomposable summand of rank 1 or of rank 3.
\end{thm}
\begin{proof}
	
	If there is a nonorientable base in $M$ with $g<-1$, then $H^\ast(M)$ will have a rank 1 summand.  Henceforth, assume all bases in $M$ have $g\geq -1$. By Section~\ref{section:cohomology} there is a basis for $H_1(M)$ consisting of elements of three types: $\alpha$, $\beta$, and $t$.
	
	If there are no elements of type $t$, then Theorem~\ref{thm:IntersectionRing} implies $H^\ast(M) \cong H^\ast\left(\#^6 S^1 \times S^2 \right)$, which splits six rank 1 summands. There must be an even number of $t$ elements as the $\alpha$'s and $\beta$'s come in pairs.
	
	If there are more than two $t$'s, there is at most one pair of $\alpha$ and $\beta$, so their image cannot contain all the $t$'s; thus, there must be a rank 1 summand generated by one of the $t$'s.  This follows from Theorem~\ref{thm:IntersectionRing}, since the span of the products among $\alpha$'s and $\beta$'s is bounded above by the number of $\alpha$-$\beta$ pairs.
	
	We may now assume that there are two $t$'s, i.e. $H_2(M)$ has a basis of the form $\{\alpha_1,\beta_1, \alpha_2,\beta_2, t_1, t_2\}$.  Here by lower case letters we mean elements in $H_2(M)$ which should be thought of as the intersection duals of the similarly denoted curves in the manifold.  Thus by Theorem~\ref{thm:IntersectionRing} the intersection form of $M$ is \[ k\alpha_1\beta_1 t_2 + \ell \alpha_1\beta_1 t_2 + m\alpha_2\beta_2 t_1 + n \alpha_2\beta_2 t_2 = \alpha_1\beta_1(kt_1 + \ell t_2) + \alpha_2\beta_2(mt_1 + nt_2)\] The reader will recall that this means that $\alpha_1$ intersects $\beta_1$ in the linear combination $kt_1^\ast + \ell t_2^\ast$ of the intersection duals of $t_1$ and $t_2$.  We will change the basis of $H_2(M)$ to $\{\alpha_1,\beta_1,\alpha_2,\beta_2,s_1,s_2\}$ by changing the dual basis in $H_1(M)$.
	
	If $kt_1^\ast + \ell t_2^\ast$ and $mt_2^\ast + nt_2^\ast$ are linearly independent, we set $s_1$ and $s_2$ to be the elements of $H_2(M)$ with $s_1^\ast = kt_1^\ast + \ell t_2^\ast$ and $s_2^\ast = mt_1^\ast + nt_2^\ast$; then the intersection form of $M$ is $\alpha_1\beta_1 s_1 + \alpha_2\beta_2 s_2$, which clearly is split into two rank 3 indecomposable summands.
	
	If $k=l=0$ or $m=n=0$, then a similar change of basis shows (at least) three rank 1 summands split off.
	
	We are left with the case that $kt_1 + \ell t_2$ is a multiple of $mt_1 + n t_2$, which is nonzero; in this case one may change the basis for $H_2(M)$ so that the intersection form is $(p\alpha_1\beta_1 + \alpha_2\beta_2)s_2$ (for some $p\in\bbq$), which has a rank 1 summand generated by $s_1$.

\end{proof}

\begin{proof}[Proof of Theorem~\ref{thm:TreeGraphManifoldsNotGeneric}]
	We exhibited in Example~\ref{ex:Awesome} a 3-manifold $M$ whose $H^\ast(M)$ splits no rank 1 summand. By Theorem~\ref{thm:splitsnoz3}, $M$ has no rank 3 indecomposable summand. By Theorem~\ref{thm:TreeSplitsOddly}, $M$ cannot be homology cobordant to a tree graph manifold.
\end{proof}

\bibliographystyle{amsalpha}
\bibliography{/Users/pdhorn/Documents/Research/PeterHornBib}

\end{document}